\documentclass{article}

\usepackage[all]{xy}
\usepackage[english]{babel}
\usepackage{amsmath}
\usepackage{amsthm}
\usepackage{amssymb}
\usepackage{graphicx}
\usepackage{verbatim}

\newtheorem{thm}{Theorem}

\newtheorem{lem}{Lemma}
\newtheorem{cor}{Corollary}

\newtheorem{prop}{Proposition}

\newcommand{\leftsub}[2]{{}_{#1}{#2}}   

\newcommand{\s}{\scriptscriptstyle}

\newcommand{\IH}{\mathbb{H}}
\newcommand{\sIH}{{\scriptscriptstyle\mathbb{H}}}

\newcommand{\IC}{\mathbb{C}}
\newcommand{\sIC}{{\scriptscriptstyle\mathbb{C}}}

\newcommand{\IN}{\mathbb{N}}
\newcommand{\INo}{\mathbb{N}_{0}}

\theoremstyle{definition}
\newtheorem{Def}{Definition}

\theoremstyle{remark}
\newtheorem{rem}{Remark}

\theoremstyle{remark}

\def\cM{\mathcal M}

\def\cP{\mathcal P}

\def\cB{\mathcal B}

\def\bR{\mathbb R} 

\def\bN{\mathbb N}

\def\bC{\mathbb C}

\def\bH{\mathbb H}

\def\bB{\mathbb B}

\def\Re{\operatorname{Re}}

\def\Im{\operatorname{Im}}

\def\fai{\varphi}
\def\La{\Lambda}
\def\la{\lambda}
\def\lz{\langle}
\def\pz{\rangle}
\def\pa{\partial}
\def\nad{\overline}
\def\pod{\underline}

\def\so{\mathfrak{so}}

\def\gog{\mathfrak{g}}

\def\mand{\text{\ \ \ and\ \ \ }}

\begin{document}

\title{The Gelfand-Tsetlin bases for spherical monogenics in dimension 3}

\author{S. Bock, K. G\"urlebeck,
R. L\' avi\v cka and V. Sou\v cek}

\maketitle

\begin{abstract}
The main aim of this paper is to recall the notion of the Gelfand-Tsetlin bases (GT bases for
short) and to use it for an explicit construction of orthogonal bases for the spaces of spherical
monogenics (i.e., homogeneous solutions of the Dirac or the generalized Cauchy-Riemann equation,
respectively) in dimension 3. In the paper, using the GT construction, we obtain explicit
orthogonal bases for spherical monogenics in dimension 3 having the Appell property and we compare
them with those constructed by the first and the second author recently (by a direct analytic
approach).


\end{abstract}

\section{Introduction}

The main aim of this paper is to discuss explicit constructions of orthogonal bases for the spaces
of spherical monogenics (i.e., homogeneous solutions of the Dirac or the generalized Cauchy-Riemann
equation, respectively) mainly in dimension 3. The theory of solutions to the Dirac or to the Cauchy-Riemann
operator can be seen at the same time as generalization of the (one-dimensional) complex function theory
 as well as refinement of harmonic analysis. Both function classes share many properties with each other
and are quite analogous to the complex case. The theory for the solutions of the Cauchy-Riemann operator
contains the concept of hypercomplex derivability whereas in the case of the Dirac equation due to the full
 rotational invariance of the solutions more tools from harmonic analysis find a direct application.

To construct orthogonal bases for spaces of
solutions of differential equations is, in general, a difficult problem. We show in the first part of the paper that the approach
formulated by Gelfand  and Tsetlin makes a construction of orthogonal bases easier in case of the Dirac equation.

The notion of a Gelfand-Tsetlin basis (GT basis) was formulated for irreducible (finite
dimensional) modules over a~general classical simple Lie algebra $\gog$ (see \cite{GT} for the
original paper and \cite{mol} for a review paper with many further citations). The main problem
solved in \cite{GT} was to write down matrices representing basis elements of $\gog$ with respect
to the GT basis. In the case when an irreducible $\gog$-module is realized explicitly (usually as a
subspace of the space of solutions of invariant differential equations), it is often possible to
construct its GT basis in quite algorithmic way. The main advantage of GT bases for practical
applications is the fact that the GT bases are automatically orthogonal with respect to any
invariant inner product on the given irreducible module.

The problem of constructing basis functions in spaces of monogenic functions has a long history. In
the very beginning it was the task to construct sufficiently many concrete monogenic functions. Already the
work of R.~Fueter contains the idea to consider a special kind of homogeneous monogenic polynomials as generalization of
the complex powers $ z^n $ and to look for an analogue of the Taylor series expansions. The result
was a series expansion in Fueter polynomials \cite{Fue1}. The important progress compared with the real
Taylor series expansion for real analytic functions was the possibility to express the increment of a
quaternion-valued functions by the hypercomplex increment of the arguments. Much later in \cite{BDS} these series were
reinvented and in \cite{Malonek1987} connected with the problem of hypercomplex derivability.
Finally, it could be shown that for Clifford algebra valued functions the existence of a local
Taylor series expansion in the symmetric powers \cite{Malonek1987}, the hypercomplex derivability
and the monogenicity are equivalent, which is a very comfortable situation and advantageous for
the solution of more complicated differential equations by means of monogenic functions. With the needs of
numerical approximations, motivated also by geometrical properties and invariance properties, a
construction of simple orthogonal systems of monogenic polynomials was needed. These problems were
connected with the idea of the Fischer decomposition (originally in  the paper \cite{fischer}) and
with the so called Almansi decompositions (see citations in \cite{Malonek-Ren}). The main
disadvantage of the Fueter polynomials for numerical purposes was that they are not orthogonal with
respect to $L_2$-inner product. That is why it was not possible to relate Taylor and Fourier
expansions so easily as in the complex case, i.e., to relate the local and the global behaviour of the functions.
First explicit constructions of complete orthonormal polynomial systems in the important case of
dimension 3 were done by I. Ca\c c\~ao \cite{cac}, the first and the second author and H.~Malonek
\cite{CacGueMal}, \cite{BockCacGue}, \cite{CacGueBock}. Main idea was the application of the
Cauchy-Riemann operator to an orthogonal system of spherical harmonics and an explicit
orthonormalization of the resulting system. These results were the basis for Fourier expansions and
related applications like the definition of a continuous operator of monogenic primitivation in the $L_2$-space of monogenic functions.

Furthermore, in \cite[pp. 254-264]{DSS} and \cite{som,van, step2}, another constructions of orthogonal bases for spherical monogenics even in all dimensions are explained.
In particular,
in \cite[Theorem 2.2.3, p. 315]{DSS}, the so-called Cauchy-Kovalevskaya (CK) method has already been developed.
But this method is not used in \cite{DSS} for a~construction of orthogonal bases although the construction is obvious
not only in dimension 3 but in an arbitrary dimension as we explain in Section \ref{sCK}. Actually, in this paper we use
the CK method for an explicit construction of the GT bases for spherical monogenics in dimension 3.
In \cite{lavSL2}, the GT bases for this case are obtained in quite a~different way and, in particular,
simple expressions of elements of these bases in terms of the Legendre polynomials are given there.
By the way, the Cauchy-Kovalevskaya method is applicable in other settings as well, see \cite{GTinH, GT2H, ckH, kerH} and \cite{DLS4}.
Similar questions were also considered by R. Delanghe for the Riesz system, see \cite{del07} and also \cite{zei}, \cite{mor09}.

Looking back at the complex case we observe that the basis functions for Taylor and Fourier
expansions are principally the same, they are real multiples of each other. An important property
of this basis is the so-called Appell property of the system $\{z^n\}_{n \in \mathbb{N}}$ with respect to the complex derivative.
Originally, P.~Appell introduced in \cite{Appell} polynomials with the property that $ \frac{d}{d_x} P_n (x) =
n \, P_{n - 1} (x)$. This property makes it possible to differentiate and integrate power series
expansions easily summand by summand and to obtain immediately a series of the same structure.
Later on Sheffer \cite{Sheffer} invented generating functions to construct Appell systems or Appell
sequences and depending on the interests of the authors nowadays one of these approaches is
preferred.\\*[2ex]
The generalization of the Appell idea to monogenic polynomials (as solutions of the Cauchy-Riemann equations) requires the
correct understanding of the hypercomplex derivative (see \cite{S}, \cite{MS} and
\cite{Gurlebeck1999}). First Appell systems of paravector-valued monogenic polynomials could be constructed by
H.~Malonek et. al. \cite{CM07}, \cite{FCM}, \cite{FM}. These systems were orthogonal but not
complete with respect to $L_2$-inner product and it was observed that the system coincides also
with a system of "special monogenic functions" as constructed in \cite{AbulConstales} without
mentioning the Appell property. In \cite{NGue2009} it was shown that the same Appell
system can be obtained by the Fueter-Sce extension of the complex Appell system $\{z^n\}_{n\in\mathbb{N}}$.
In \cite{CM08}, I. Ca\c c\~ao and H. Malonek constructed an
orthogonal Appell basis in $L_2$, equipped with the real inner product, for the solutions of the
Riesz system in dimension $3$. Later on, in a series of papers \cite{BG}, \cite{Bock2009},
\cite{Bock2010a} the first and the second author elaborate an orthogonal Appell basis of monogenic
polynomials for the space of square integrable solutions of the Cauchy-Riemann system in $\mathbb{R}^3$ (Moisil-Teodorescu system) with
respect to the quaternion-valued inner product. In \cite{Bock2009}, this system was used to
approximate solutions of the Lam\'{e} - Navier equations of linear elasticity theory.

Important for practical applications is also that this Appell system can be defined recursively
(see \cite{Bock2010a} and Theorem \ref{BG} below) and that it is not longer necessary to start with spherical harmonics.

The question arises if this system is only one that fortunately could be constructed or if it is
unique (in a certain sense). Because of the increasing amount of calculations it becomes important
to understand the underlying general principle of the constructions, to find a way to construct
bases in all dimensions. First results were obtained in \cite{Bock2010c} where a unified and
explicit construction principle of monogenic Appell bases in dimension 2, 3 and 4 was proved.

In low dimensions (3 or 4), it is quite common to consider quaternion valued functions instead of spinor valued ones,
and to replace complex vector spaces of solutions with vector spaces over the skew field of real quaternions. Analyzing
all the mentioned concrete results on Appell systems of monogenic polynomials and relating them to the case of the Dirac equation it becomes visible that there is some general
scheme in the background - the so-called Gelfand-Tsetlin bases. It is possible to relate both picture, and we shall do it below.

In the paper, we apply a general scheme of GT bases to the case of spherical monogenics in
dimension 3 and we write down explicit formulae for the corresponding orthogonal GT bases in terms
of spinor valued and quaternion valued functions. The elements of the obtained bases can be easily
renormalized to have the Appell property. Actually, it turns out that such an requirement is
characterizing the bases uniquely  (see Theorem \ref{appell} below). We compare then the formulae
obtained for quaternion valued functions with those obtained by the first and the second author in
\cite{BG} and we show that they coincide.

In Section 2, we start with a short summary of notation needed to formulate a general construction
of the GT bases. In Section 3, we show that the branching rules needed to perform the construction
of the GT bases explicitly can be realized using only classical tools of Clifford analysis, namely,
the Fischer decomposition and the Cauchy-Kovalevskaya extension. Actually, we just apply the
Cauchy-Kovalevskaya method developed already in \cite[Theorem 2.2.3, p. 315]{DSS}. In the rest of
this paper, we study properties of GT bases mainly in dimension 3. A detailed study of GT bases in
higher dimensions will be given in a~next paper. An explicit construction of the GT bases in
dimension 3 is written down in Section 4, see Theorem \ref{GT3} and Corollary \ref{corGT3}. To do
it, we use the Fischer decomposition in dimension 2 in the same way as it is done in higher
dimensions. Let us remark that the Fischer decomposition in dimension 2 (see Theorem \ref{decomp2})
is not usually considered in Clifford analysis and it has a slightly different form than in higher
dimensions. In particular, we show that the GT bases for spinor valued spherical monogenics in
dimension 3 possess a~generalization of the Appell property, that is, they possess an Appell
property not only w.r.t.\ the last real variable $x_3$ but also w.r.t.\ the remaining complex
variables $z$ and $\nad z,$ see Corollary \ref{corGT3}. Finally, in Section 5, we introduce the
quaternionic formulation and we describe its relation to the spinor case. We reformulate the GT
bases in quaternionic language (see Theorem \ref{appell} and Corollary \ref{corappell} below) and
we show that the bases having the Appell property coincide with those constructed by the first and
the second author in \cite{BG} for the Cauchy-Riemann system. This system has the Appell property
 with respect to the hypercomplex derivative on the basis polynomials orthogonal to the
 hyperholomorphic constants and then with respect to a complex derivative on the remaining basis functions.
In the end of the paper we present some applications of both approaches and construct new Taylor
 series and Fourier series expansions, respectively.

\section{Preliminaries}

First we introduce some notation. Let $(e_1,\ldots,e_m)$ be the standard basis of the Euclidean space $\bR^m$
and let $\bC_m$ be the complex Clifford algebra generated by the vectors $e_1,\ldots,e_m$ such that
$e_j^2=-1$ for $j=1,\ldots,m.$
As usual, we identify a~vector $x=(x_1,\ldots,x_m)\in\bR^m$ with the element $x_1e_1+\cdots+x_me_m$ of $\bC_m.$
Recall that the Spin group $Spin(m)$ is defined as the set of products of even number of unit vectors of $\bR^m$ endowed with the Clifford multiplication.
Now we introduce spaces of spherical monogenics. For a~vector space $V,$ we denote by $\cP_k(\bR^m,V)$ the space of $V$-valued polynomials in $\bR^m$ which are homogeneous of degree $k.$
Let $S$ be a~subspace of $\bC_m$ invariant with respect to the left multiplication by elements of $Spin(m).$ Then put
\begin{equation}
\label{monogenics}
\cM_k(\bR^m,S)=\{P\in\cP_k(\bR^m,S):\pa P=0\}
\end{equation}
where the Dirac operator $\pa$ in $\bR^m$ is defined as
$$\pa=e_1\frac{\pa\ }{\pa x_1}+\cdots+e_m\frac{\pa\ }{\pa x_m}.$$
It is well-known that if $S$ is a~basic spinor representation of the group $Spin(m)$ then the space $\cM_k(\bR^m,S)$ of spherical monogenics is an irreducible module under the so-called $L$-action,
defined by
$$
[L(s)(P)](x) = s\,P(s^{-1}xs),\ s\in Spin(m)\text{\ \ and\ \ }x=(x_1,\ldots,x_m)\in\bR^m.
$$

In this paper, we are interested in a~construction of GT bases of spherical monogenics.
Let us recall briefly the concept of  GT bases for the orthogonal case, see \cite{mol, GT}. In what follows, we
deal with complex representations of the Lie algebra $\so(m)$ of the Spin group $Spin(m).$
Let us consider a~general irreducible $\so(m)$-module $V(\mu_m)$ with the highest weight $\mu_m.$
In the even dimensional case $m=2n,$ the highest weight $\mu_m$ is a~vector $$\mu_m=(\la_{m,1},\ldots,\la_{m,n})$$ consisting entirely of integers or entirely of non-zero half-integers which satisfy the relation
\begin{equation}
\label{evenweight}
\la_{m,1}\geq\la_{m,2}\geq\cdots\geq\la_{m,n-1}\geq|\la_{m,n}|.
\end{equation}
In the odd dimensional case $m=2n+1,$ the vector $\mu_m=(\la_{m,1},\ldots,\la_{m,n})$ satisfies instead the condition
\begin{equation}
\label{oddweight}
\la_{m,1}\geq\la_{m,2}\geq\cdots\geq\la_{m,n}\geq 0.
\end{equation}
Furthermore, as is well known, the Lie algebra $\so(m)$ can be realized as the space of bivectors of Clifford algebra $\bC_m.$
In what follows, we consider a~chain of Lie algebras
\begin{equation}\label{chain}
\so(m)\supset\so(m-1)\supset\cdots\supset\so(2)
\end{equation}
where, for $k=2,\ldots,m,$ $$\so(k)=\lz \{e_{ij}: 1\leq i<j\leq
k\}\pz.$$ Here $e_{ij}=e_ie_j$ and $\lz M\pz$ stands for the span of
a~set $M.$

The key ingredient for introduction of a~GT basis is the following branching rule well-known in representation theory:
As an $\so(m-1)$-module, the given module $V(\mu_m)$
decomposes into a~multiplicity free direct sum of irreducible
$\so(m-1)$-modules
\begin{equation}\label{branch}
V(\mu_m)=\bigoplus_{\mu_{m-1}}V(\mu_m,\mu_{m-1})
\end{equation}
where the direct sum is taken over the highest weights $\mu_{m-1}$ satisfying the conditions \eqref{evenbranch} and \eqref{oddbranch} below.
Moreover, it is well-known that if the weight $\mu_m$ consists entirely of
non-zero half-integers (or integers), then so do all highest weights $\mu_{m-1}.$
In the case when $m=2n,$ the direct sum (\ref{branch}) is taken
over all highest weights
$\mu_{m-1}=(\la_{m-1,1},\ldots,\la_{m-1,n-1})$ such that
\begin{equation}
\label{evenbranch}
\la_{m,1}\geq\la_{m-1,1}\geq\la_{m,2}\geq\cdots\geq\la_{m,n-1}\geq\la_{m-1,n-1}\geq|\la_{m,n}|.
\end{equation}
In the case when $m=2n+1,$ the direct sum (\ref{branch}) is taken over all highest weights $\mu_{m-1}=(\la_{m-1,1},\ldots,\la_{m-1,n})$ such that
\begin{equation}
\label{oddbranch}
\la_{m,1}\geq\la_{m-1,1}\geq\la_{m,2}\geq\cdots\geq\la_{m,n-1}\geq\la_{m-1,n-1}\geq\la_{m,n}\geq |\la_{m-1,n}|.
\end{equation}
Moreover, with respect to any given invariant inner product on the module $V(\mu_m),$ the decomposition (\ref{branch}) is even orthogonal.

Of course, we can decompose further each module $V(\mu_m,\mu_{m-1})$ of the decomposition (\ref{branch}) into irreducible $\so(m-2)$-modules $V(\mu_m,\mu_{m-1},\mu_{m-2})$ and so on. Hence we end up with the decomposition of the given $\so(m)$-module
$V(\mu_m)$ into irreducible $\so(2)$-modules $V(\mu).$
Moreover, any such module $V(\mu)$ is uniquely determined by the so-called Gelfand-Tsetlin pattern
\begin{equation}
\label{pattern}
\mu=(\mu_m,\mu_{m-1},\ldots,\mu_2).
\end{equation}
Here $\mu$ as in (\ref{pattern}) is called the Gelfand-Tsetlin
pattern provided that each vector $\mu_j$ satisfies the conditions
(\ref{evenweight})-(\ref{oddbranch}) (with $m$ replaced by $j$) and
the numbers $\la_{j,k}$ are either all integers or all non-zero
half-integers. We denote by $P(\mu_m)$ the set of the
Gelfand-Tsetlin patterns whose first term is the highest weight
$\mu_m.$ To summarize, we decompose the given module $V(\mu_m)$ into
the direct sum of irreducible $\so(2)$-modules
\begin{equation}\label{branch+}
V(\mu_m)=\bigoplus_{\mu\in P(\mu_m)}V(\mu).
\end{equation}
Moreover, the decomposition (\ref{branch+}) is obviously orthogonal.
Let us note that the decomposition \eqref{branch+} is uniquely specified by the choice of the chain of Lie subalgebras
\eqref{chain}.

Since all submodules $V(\mu)$ are, in fact, one-dimensional we obtain easily an orthogonal basis of $V(\mu_m)$ by taking a~non-zero vector $e(\mu)$ from each module $V(\mu).$
The orthogonal basis $$E=\{e(\mu):\mu\in P(\mu_m)\}$$ is then called a~GT basis of the module $V(\mu_m).$
It is easily seen that, by the definition, the vector $e(\mu)$ is uniquely determined by $\mu\in P(\mu_m)$ up to a~scalar multiple.

\section{The Cauchy-Kovalevskaya method}\label{sCK}

To construct a~GT basis for the $\so(m)$-module $\cM_k(\bR^m, S)$ it is clear that we
need to describe quite explicitly the branching rule \eqref{branch} for this module, that is, its decomposition into irreducible
$\so(m-1)$-submodules. To this end we use only two basic
tools from Clifford analysis, namely, the Cauchy-Kovalevskaya
extension and the Fischer decomposition of spinor-valued
polynomials.
Actually, we just apply the Cauchy-Kovalevskaya method developed already in \cite[Theorem 2.2.3, p. 315]{DSS}.
We first state the Fischer decomposition, see \cite[p.
206]{DSS}.

\begin{prop}\label{fischer}
Let $m\geq 3$ and let $S$ be a~spinor space of the Clifford algebra $\bC_m,$ that is, $S$ is an
irreducible (left) module over $\bC_m.$  Then
$$\cP_k(\bR^m,S)=\bigoplus_{j=0}^kx^j\cM_{k-j}(\bR^m,S).$$
\end{prop}

\begin{rem}
An analogous decomposition is valid also in the dimension $m=2,$ see Theorem \ref{decomp2} below for details.
\end{rem}

Now we recall the Cauchy-Kovalevskaya extension.
Let $p$ be a~$k$-homogeneous polynomial in $\bR^m$ which takes values in a~spinor space $S$ of $\bC_m.$
Such a~polynomial $p$ can be uniquely expressed as
$$p(x)=\sum_{j=0}^kp_j(\pod x)\;x_m^j$$
where $p_j$ is an $S$-valued polynomial in $\pod
x=(x_1,\ldots,x_{m-1})\in \bR^{m-1}$ which is homogeneous of degree
$k-j.$ Moreover, putting $$\pod\pa=e_1\frac{\pa\ }{\pa
x_1}+\cdots+e_{m-1}\frac{\pa\ }{\pa x_{m-1}},$$ it is easy to see
that the Dirac equation $\pa p=0$ holds if and only if, for each
$j=0,\ldots,k,$
$$p_j=\frac 1j (e_m\pod\pa)\; p_{j-1}=\cdots=\frac 1{j!}(e_m\pod\pa)^j p_0.$$
In this case, we have thus that
$$p(x)=\sum_{j=0}^k\frac 1{j!}(e_mx_m\pod\pa)^j p_0(\pod x)=(e^{e_mx_m\pod\pa}p_0)(x).$$
Now it is easy to obtain the following result, see \cite[p. 152]{DSS}.

\begin{prop}\label{ck}
Let $S$ be a~basic spinor representation of the group $Spin(m).$
Then the Cauchy-Kovalevskaya extension operator
$$CK=e^{e_mx_m\pod\pa}$$ is an $\so(m-1)$-invariant isomorphism of the module $\cP_k(\bR^{m-1},S)$ onto the module $\cM_k(\bR^m,S).$
\end{prop}

As we explain later on, to describe explicitly the branching rules in our situation we need to understand the CK extension of
particular terms in the Fischer decomposition, that is, the CK extension of
polynomials
of the form $\pod x^jp(\pod x)$ with $p$ being a~spherical monogenic.
But first recall that the Gegenbauer polynomial $C^{\nu}_j$ is defined as
\begin{equation}\label{gegenbauer}
C^{\nu}_j(z)=\sum_{i=0}^{[j/2]}\frac{(-1)^i(\nu)_{j-i}}{i!(j-2i)!}(2z)^{j-2i}\text{\ \ with\ \ }
(\nu)_{j}=\nu (\nu+1)\cdots (\nu+j-1),
\end{equation}
see \cite[p. 302]{AAR}.

\begin{lem}\label{lckxj}
Let $j\in\bN_0$ and $p\in\cM_k(\bR^{m-1}, S).$ Then we have that
$$CK((\pod xe_m)^j p(\pod x))=X^{(j)} p(\pod x)$$
where $X^{(0)}=1$ and, for $j\in\bN,$ the polynomial $X^{(j)}=X^{(j)}_k$ is given by
$$X^{(j)}_k(\pod x,x_m)=
\mu^j_kr^j\left (C_j^{m/2+k-1}(\frac{x_m}{r})+\frac{m+2k-2}{m+2k+j-2}C_{j-1}^{m/2+k}(\frac{x_m}{r})\frac{\pod xe_m}{r}\right )
$$
with $r=(x_1^2+x^2_2+\cdots+x_m^2)^{1/2},$ $\mu^{2l}_k=(-1)^l(C_{2l}^{m/2+k-1}(0))^{-1}$ and $$\mu^{2l+1}_k=(-1)^l\frac{m+2k+2l-1}{m+2k-2}(C_{2l}^{m/2+k}(0))^{-1}.$$
\end{lem}

\begin{proof}
In \cite[p. 312, Theorem 2.2.1]{DSS}, the corresponding polynomial we denote here by $\tilde X^{(j)}_k$ is computed for the Cauchy-Riemann operator. Fortunately,
there is an obvious relation between these two polynomials. Namely, we have that
\begin{equation*}
X^{(j)}_k(\pod x,x_m)=\left\{
\begin{array}{ll}
\tilde X^{(j)}_k(\pod xe_m,x_m),& j\text{\ even},\medskip\\{}
-\tilde X^{(j)}_k(\pod xe_m,x_m) e_m,& j\text{\ odd}.
\end{array}
\right.
\end{equation*}
To complete the proof it is sufficient to use the explicit formula for the polynomial $\tilde X^{(j)}_k.$
\end{proof}

At this moment we are ready to describe the decomposition of the $\so(m)$-module $\cM_k(\bR^m, S)$ into irreducible $\so(m-1)$-submodules.
We start with the even dimensional case.

\paragraph{The even dimensional case}

In the case when $m=2n,$
there is a~unique (up to equivalence) irreducible module $S_m$ over $\bC_m.$
As a~$Spin(m)$-module, $S_m$ is reducible and decomposes into two inequivalent irreducible submodules
$$S_m=S^+_m\oplus S^-_m.$$
Actually, $S^{\pm}_{2n}$ are unique basic spinor representations of the group $Spin(2n)$
and, putting $\theta_{2n}=(-i)^ne_1e_2\cdots e_{2n},$ we have that
\begin{equation}
\label{spm}
S^{\pm}_{2n}=\{u\in S_{2n}: \theta_{2n}u=\pm u\}.
\end{equation}
Furthermore, as $Spin(2n-1)$-modules, $S^+_{2n}$ and $S^-_{2n}$ remain still irreducible but become equivalent to each other.

Let $S$ be a~basic spinor representation for $Spin(2n),$ that is, $S\simeq S^+_{2n}$ or $S\simeq S^-_{2n}.$
In any case, it is easy to see that Proposition \ref{ck} implies that
$$\cM_k(\bR^{2n},S)=CK(\cP_k(\bR^{2n-1},S).$$
Moreover, using Proposition \ref{fischer}, we get the following decompositions of the spaces $\cP_k(\bR^{2n-1},S)$
into inequivalent irreducible $\so(2n-1)$-submodules
$$\cP_k(\bR^{2n-1},S)=
\bigoplus_{j=0}^k \;(\pod xe_{2n})^j\cM_{k-j}(\bR^{2n-1},S).$$

Finally, applying the CK extension to this decomposition and using Lemma \ref{lckxj}, we get obviously the next result, cf.\ \cite[Theorem 2.2.3, p. 315]{DSS}.

\begin{thm}\label{evendecomp}
Let $n\geq 2$ and let $S$ be a~basic spinor representation for $Spin(2n).$
Then the $\so(2n)$-module $\cM_k(\bR^{2n},S)$ decomposes into
inequivalent irreducible $\so(2n-1)$-submodules as
$$\cM_k(\bR^{2n},S)=
\bigoplus_{j=0}^k X^{(j)}\cM_{k-j}(\bR^{2n-1},S).$$
\end{thm}

Of course, using Theorem \ref{evendecomp}, it is easy to construct GT bases in dimension $2n$ when we know GT bases in dimension $2n-1.$

\begin{cor}\label{GTeven}
Let $\cB^{2n-1}_j(S)$ be GT bases of the modules $\cM_j(\bR^{2n-1},S)$ for all $j=0,\ldots,k.$
Then we have that the set
$$\cB^{2n}_k(S)=
\bigcup_{j=0}^k X^{(j)}\cB^{2n-1}_{k-j}(S)$$
is a~GT basis of the module $\cM_k(\bR^{2n},S).$
Here the polynomial $X^{(j)}$ is defined as in Lemma \ref{lckxj} and, of course, we put
$$X^{(j)}\cB^{2n-1}_{k-j}(S)=\{X^{(j)}p\ |\ p\in\cB^{2n-1}_{k-j}(S)\}.$$
\end{cor}

Now we are going to deal with the odd dimensional case.

\paragraph{The odd dimensional case}

In the case when $m=2n+1,$ there are just two different irreducible $\bC_m$-modules (equivalent to) $S^{\pm}_{m+1}.$
On the other hand,
there exists only a~unique  basic spinor representation $S$ of the group $Spin(m).$
In particular, as $Spin(m)$-modules, the modules $S^{\pm}_{m+1}$ are both equivalent to $S.$
Moreover, $S$ can be viewed also as an irreducible $\bC_{2n}$-module, that is, $S\simeq S_{2n}.$ As we know (see (\ref{spm})), we have therefore that
$S=S^+\oplus S^-$
where
$$
S^{\pm}=\{u\in S: \theta_{2n}u=\pm u\}
$$
are both irreducible $Spin(2n)$-modules.

Furthermore, according to Proposition \ref{ck}, we have that
$$\cM_k(\bR^m,S)=CK(\cP_k(\bR^{m-1},S)).$$
By Proposition \ref{fischer}, we can easily obtain the following decomposition of the space $\cP_k(\bR^{m-1},S)$
into inequivalent irreducible $\so(m-1)$-submodules
$$\cP_k(\bR^{m-1},S)=\bigoplus_{j=0}^k\; (\pod x e_m)^j\cM_{k-j}(\bR^{m-1},S^+)\oplus (\pod x e_m)^j\cM_{k-j}(\bR^{m-1},S^-).$$
Applying the CK extension to this decomposition together with Lemma \ref{lckxj} gives the following result, cf.\ \cite[Theorem 2.2.3, p. 315]{DSS}.

\begin{thm}\label{odddecomp}
Let $n\geq 2$ and let $S$ stand for a~basic spinor representation of $Spin(2n+1).$
Then the $\so(2n+1)$-module $\cM_k(\bR^{2n+1},S)$
decomposes into inequivalent irreducible $\so(2n)$-submodules as
follows:
$$\cM_k(\bR^{2n+1},S)=\bigoplus_{j=0}^k\; X^{(j)}\cM_{k-j}(\bR^{2n},S^+)
\oplus X^{(j)}\cM_{k-j}(\bR^{2n},S^-).$$
\end{thm}

\begin{cor}\label{GTodd}
Let $\cB^{2n}_j(S^{\pm})$ be GT bases of the modules $\cM_j(\bR^{2n},S^{\pm})$ for all $j=0,\ldots,k.$
Then we have that the set
$$\cB^{2n+1}_k(S)=
\bigcup_{j=0}^k\; X^{(j)}\cB^{2n}_{k-j}(S^+)\cup X^{(j)}\cB^{2n}_{k-j}(S^-)$$
is a~GT basis of the module $\cM_k(\bR^{2n+1},S).$
Here the polynomial $X^{(j)}$ is defined as in Lemma \ref{lckxj}.
\end{cor}

To summarize Corollaries \ref{GTeven} and \ref{GTodd} tell us that GT bases for spherical monogenics can be obtained inductively. Indeed, whenever we know GT bases in dimension $m-1$ we can easily construct GT bases in dimension $m.$

\section{The Gelfand-Tsetlin bases in dimension 3}

In this section, we construct explicitly GT bases for spinor valued spherical monogenics in dimension 3. First we recall a~realization of basic spinor representations $S^{\pm}_{2n}.$

\paragraph{Basic spinor representations $S^{\pm}_{2n}$}
For $j=1,\ldots,n,$ put
$$w_j=\frac 12(e_{2j-1}+ie_{2j}),\ \ \nad w_j=\frac 12(-e_{2j-1}+ie_{2j})\text{\ \ and\ \ }I_j=\nad w_jw_j.$$
Then $I_1,\ldots,I_n$ are mutually commuting idempotent elements in $\bC_{2n}.$
Moreover, $I=I_1I_2\cdots I_n$ is a~primitive idempotent in $\bC_{2n}$ and
$$S_{2n}=\bC_{2n}I$$
is a~minimal left ideal in $\bC_{2n}.$ Putting $W=\lz w_1,\ldots,w_n\pz,$ we have that
$$S_{2n}=\La(W)I,\ \  S^{+}_{2n}=\La^{+}(W)I\text{\ \ and\ \ } S^{-}_{2n}=\La^{-}(W)I$$
where $\La(W)$ is the exterior algebra over $W$ with the even part $\La^{+}(W)$ and the odd part $\La^{-}(W).$
See \cite[pp. 114-118]{DSS} for details.

Furthermore, it is well-known that, for each $u\in\bC_{2n},$ there is a~unigue complex number $[u]_0$ such that
$IuI=[u]_0I$ and that an inner product on $S_{2n}$ is given by
\begin{equation}\label{vacuum}
(s,t)=[\nad u v]_0\text{\ for\ }s=uI,\ t=vI\text{\ with\ }u,v\in\bC_{2n}.
\end{equation}
Here, for each Clifford number $u\in\bC_m,$ $\nad u$ stands for its Clifford conjugate.
See \cite[pp. 120-125]{DSS} for details.

In the next paragraph, we introduce invariant inner products on the spin modules of spherical monogenics.

\paragraph{Invariant inner products}

Let us remark that, on each (finite-dimensional) irreducible representation of $Spin(m)$ there exists always an invariant inner product and, in addition, that the invariant inner product is determined uniquely up to a~positive multiple. In what follows, we recall two well-known realizations of the invariant inner product on the module $\cM_k(\bR^m,S),$
namely, the $L_2$-inner product and the Fischer inner product. For $P,Q\in \cM_k(\bR^m,S),$ we define the $L_2$-inner product of $P$ and $Q$ as
\begin{equation}\label{L2product}
(P,Q)_1=\int_{\bB_m}(P,Q)\;d\la^m
\end{equation}
where $\bB_m$ is the unit ball in $\bR^m$ and $d\la^m$ is the Lebesgue measure in $\bR^m.$

Now we introduce the Fischer inner product. Each $P\in\cP_k(\bR^m,S)$ is of the form
$$P(x)=\sum_{|\alpha|=k}a_{\alpha}x^{\alpha}$$
where the sum is taken over all multi-indexes $\alpha=(\alpha_1,\ldots,\alpha_m)$ of $\bN^m_0$ with  $|\alpha|=\alpha_1+\cdots+\alpha_m=k,$ all coefficients $a_{\alpha}$ belong to $S$ and $x^{\alpha}=x_1^{\alpha_1}\cdots x_m^{\alpha_m}.$ For $P,Q\in\cP_k(\bR^m,S),$ we define the Fischer inner product of $P$ and $Q$ as
\begin{equation}\label{Fischerproduct}
(P,Q)_2=\sum_{|\alpha|=k}\alpha!\;(a_{\alpha},b_{\alpha})
\end{equation}
where $\alpha!=\alpha_1!\cdots\alpha_m!,$ $P(x)=\sum a_{\alpha}x^{\alpha}$ and $Q(x)=\sum b_{\alpha}x^{\alpha}.$
It is easily seen that
$$(P,Q)_2=[(\nad P(\frac{\pa\ }{\pa x})Q)(0)]_0
\text{\ \ \ with\ \ \ }\nad P(\frac{\pa\ }{\pa
x})=\sum_{|\alpha|=k}\nad a_{\alpha}\frac{\pa^{|\alpha|}}{\pa
x^{\alpha}}.$$ Here $\pa^{|\alpha|}/\pa
x^{\alpha}=(\pa^{\alpha_1}/\pa
x^{\alpha_1}_1)\cdots(\pa^{\alpha_m}/\pa x^{\alpha_m}_m)$ as usual.

\paragraph{Fischer decompositions in the dimension $m=2$}
As we have remarked in Introduction, the Fischer decomposition in dimension 2 is not usually considered in Clifford analysis and it has a slightly different form than in higher dimensions.
In this case, we have that $\so(2)=\lz e_{12}\pz,$ $S=S_2=\lz I_1, w_1I_1\pz,$ $S^{+}=\lz I_1\pz$ and $S^{-}=\lz w_1I_1\pz$ with
$$I_1=\frac 12(1-ie_{12})\text{\ \ \ and\ \ \ }w_1I_1=\frac 12(e_1+ie_2).$$
Each $s\in S$ is of the form $s=s^{+}I_1+s^{-}w_1I_1$ for some complex numbers $s^{\pm}.$ We write $s=(s^{+},s^{-}).$
Let us remark that each $P\in\cP_k(\bR^2,S)$ can be expressed as $P=(P^{+},P^{-})$ for some complex valued $k$-homogeneous polynomials $P^{\pm}$ in variables $z=x_1+ix_2$ and $\nad z=x_1-ix_2.$
Furthermore, the action of $\so(2)$ on the space $\cP_k(\bR^2,S)$ is given by
$$dL(e_{12}/2)=\frac{d}{dt} L(\exp(te_{12}/2))|_{t=0}=\frac{e_{12}}{2}+x_2\frac{\pa\ }{\pa x_1}-x_1\frac{\pa\ }{\pa x_2}.$$
Put $L_{12}=dL(e_{12}/2).$
Now it is easy to show the next result.

\begin{thm}\label{decomp2} Let $\cM^{2,\pm}_{j}=\cM_j(\bR^2,S^{\pm})$ for each $j=0,\ldots,k.$ Then we have that
$\cM^{2,+}_{j}=\lz (\nad z^j,0)\pz,$ $\cM^{2,-}_{j}=\lz (0,z^j)\pz,$
$$\cP_k(\bR^2,S^{+})=\bigoplus_{j=0}^k z^j\cM^{2,+}_{k-j}
\text{\ \ \ and\ \ \ }
\cP_k(\bR^2,S^{-})=\bigoplus_{j=0}^k \nad z^j\cM^{2,-}_{k-j}.$$
In addition, for each $j=0,\ldots,k,$
the $\so(2)$-modules $z^j\cM^{2,+}_{k-j}$ and $\nad z^j\cM^{2,-}_{k-j}$ are both irreducible with the highest weights $k+\frac 12 -2j$ and $-k-\frac 12 +2j,$ respectively.
\end{thm}

\begin{proof}
Let $P\in\cP_k(\bR^2,S)$ and $P=(P^{+},P^{-}).$ Denote
$$\frac{\pa\ }{\pa z}=\frac 12(\frac{\pa\ }{\pa x_1}-i\frac{\pa\ }{\pa
x_2})\text{\ \ \ and\ \ \ } \frac{\pa\ }{\pa \nad z}=\frac
12(\frac{\pa\ }{\pa x_1}+i\frac{\pa\ }{\pa x_2}).$$ Since
$e_1P=(-P^{-},P^{+}),$ $e_{12}P=(iP^{+},-iP^{-})$ and
$\pa=e_1(\frac{\pa\ }{\pa x_1}-e_{12}\frac{\pa\ }{\pa x_2})$ we have
that $$\pa P=2(-\frac{\pa P^{-}}{\pa \nad z},\frac{\pa P^{+}}{\pa
z}).$$ Assume now that $P$ is $S^{+}$-valued, that is, $P=(P^{+},0)$
and $$P^{+}(z,\nad z)=\sum_{j=0}^k a_j z^j\nad z^{k-j}\ \
(a_j\in\bC) .$$ Obviously, $\pa P=0$ if and only if $P^{+}=a_k\nad
z^k.$ Hence it remains to show that the module $z^j\cM^{2,+}_{k-j}$
has the highest weight $k+\frac 12 -2j.$ But it follows from the
fact that weights are just eigenvalues of the operator $H=-iL_{12}$
and $$H((z^j\nad z^{k-j},0))=(k+\frac 12 -2j)(z^j\nad z^{k-j},0).$$
For $S^{-}$-valued polynomials, an analogous proof works.
\end{proof}

The decompositions of the spaces $\cP^+_k=\cP_k(\bR^2,S^+)$ are
depicted in columns of Figure \ref{fig1}. In this diagram, we write $z^j\nad
z^k$ for $(z^j\nad z^k,0).$  Moreover, all irreducible submodules
with the same highest weight are contained in the row labeled by
this highest weight.

\begin{figure}[htb]
\centerline{
\xymatrix{
         &\cP_0^+ & \cP_1^+ &\cP_2^+ &\cP_3^+ &\cP_4^+ &\\
\frac 72&                 &&& \lz \nad z^3 \pz && \cdots \\
\frac 52&          && \lz \nad z^2\pz && \lz z\nad z^3 \pz\\
\frac 32&     & \lz \nad z \pz && \lz z\nad z^2 \pz && \cdots\\
\frac 12 & \lz 1 \pz && \lz z\nad z \pz && \lz z^2\nad z^2 \pz \\
-\frac 12&     & \lz z \pz && \lz z^2\nad z \pz && \cdots\\
-\frac 32&          && \lz z^2 \pz && \lz z^3\nad z \pz\\
-\frac 52&                &&& \lz z^3 \pz && \cdots }}
\caption{The decomposition of the modules
$\cP^+_k=\cP_k(\bR^2,S^+)$.}
\label{fig1}
\end{figure}

Of course, an analogous diagram can be created for $S^-$-valued polynomials.
But, in this case, labels of rows of the diagram are shifted. In particular, the row beginning with $\lz 1 \pz$ is labeled by $-1/2.$

\paragraph{GT bases for the dimension $m=3$}
In this paragraph, we obtain explicit formulae for the GT bases of
spinor valued spherical monogenics in dimension 3.
In this case, we have that $S\simeq S^{\pm}_4,$ $\so(3)=\lz e_{12}, e_{23}, e_{31}\pz$ and
$\so(2)=\lz e_{12} \pz.$ Furthermore, the action of $\so(3)$ on the space $\cP_k(\bR^3,S)$ is given by
$$L_{ij}=dL(e_{ij}/2)=\frac{e_{ij}}{2}+x_j\frac{\pa\ }{\pa x_i}-x_i\frac{\pa\ }{\pa x_j}\ \ \ (i\not=j).$$
As a~$\so(2)$-module, the module $S$ is
reducible and decomposes into two inequivalent irreducible
submodules $S=S^{+}\oplus S^{-}$ with
$$S^{\pm}=\{u\in S: -ie_{12}\;u=\pm u\}.$$
Let $v^{\pm}$ be generators of $S^{\pm},$ that is, $S^{\pm}=\lz v^{\pm}\pz.$
We can construct a~GT basis in this case using Proposition \ref{ck}
and Theorem \ref{decomp2}.

\begin{thm}\label{GT3}
For each $k\in\bN_0,$ the polynomials
$$f^k_{2j}=e^{x_3e_3\pod\pa}\;(\frac{z^j \nad z^{k-j}}{j!(k-j)!}\;v^+)\text{\ \ and\ \ }
f^k_{2j+1}=e^{x_3e_3\pod\pa}\;(\frac{z^j \nad z^{k-j}}{j!(k-j)!}\;v^-),\ \ j=0,\ldots,k$$
form a~GT basis of the
irreducible $\so(3)$-module $\cM_k(\bR^3,S).$ Moreover, for each
$j=0,\ldots,2k+1,$ the polynomial $f_j^k$ is a~weight vector with
the weight $k+\frac 12-j,$ that is, putting $H=-iL_{12},$ we have
that $Hf_j^k=(k+\frac 12-j)f_j^k.$
\end{thm}

It is not difficult to express the GT bases from Theorem \ref{GT3} even more explicitly. To do this we identify the space $S$ with $\bC^2.$ Indeed, each $s\in S$ is of the form $$s=s^+v^++s^-v^-$$ for some complex numbers $s^+$ and $s^-.$ We write $s=(s^+,s^-)$ for short.
For the sake of explicitness, we limit ourselves to the case when $S=S^+_4$ or $S=S^-_4.$
In the former case, we put
$v^{+}=I$ and $v^{-}=w_1w_2I.$
In the latter case, we put
$v^{+}=w_2I$ and $v^{-}=w_1I.$
In these cases, explicit formulae for GT-bases are given in Corollary \ref{corGT3} below.

\begin{cor}\label{corGT3}
Let $\{f^{k,\pm}_0, \ldots, f^{k,\pm}_{2k+1}\}$ be the GT bases of
$\cM_k(\bR^3,S^{\pm}_4)$ defined in Theorem \ref{GT3}.

\medskip\noindent
(a) For each
$k\in\bN_0$ and $j=0,\ldots,k,$ we have that
$$f^{k,\pm}_{2j}=(p^k_j,\mp q^k_j)
\text{\ \ \ and\ \ \ }f^{k,\pm}_{2j+1}=(\pm q^k_{j+1},\; p^k_j)$$ where
$$p^k_j(z,\nad z, x_3)=\sum_{s=0}^{\min(j,k-j)}(-1)^s\frac{(2x_3)^{2s}\;z^{j-s}\;\nad z^{k-j-s}}{(2s)!(j-s)!(k-j-s)!}\text{\ \ \ and}$$
$$q^k_j(z,\nad z, x_3)=\sum_{s=0}^{\min(j-1,k-j)}(-1)^s\frac{(2x_3)^{2s+1}\;z^{j-1-s}\;\nad z^{k-j-s}}{(2s+1)!(j-1-s)!(k-j-s)!}.$$
Here $q^k_0=0=q^k_{k+1}.$

\medskip\noindent
(b) Moreover, for each $k\in\bN,$ we have that
\begin{eqnarray*}
\frac{\pa f^{k,\pm}_j}{\pa x_3}
&=& \left\{
\begin{array}{ll}
\mp(-1)^j 2\; f^{k-1,\pm}_{j-1},&\ \ \ j=1,\ldots,2k;\medskip\\{}
0,&\ \ \ j=0,2k+1;
\end{array}
\right.
\\
\frac{\pa f^{k,\pm}_j}{\pa z}
&=& \left\{
\begin{array}{ll}
f^{k-1,\pm}_{j-2},&\ \ \ j=2,\ldots,2k+1;\medskip\\{}
0,&\ \ \ j=0,1;
\end{array}
\right.
\\
\frac{\pa f^{k,\pm}_j}{\pa \nad z}
&=& \left\{
\begin{array}{ll}
f^{k-1,\pm}_{j},&\ \ \ j=0,\ldots,2k-1;\medskip\\{}
0,&\ \ \ j=2k,2k+1.
\end{array}
\right.
\end{eqnarray*}

\medskip\noindent
(c) Finally, for $k\in\bN_0$ and $j=0,\ldots,2k+1,$ we have that
$$f^{k,\pm}_{2k+1-j}=(-1)^j(f^{k,\pm}_j)^*$$ where $s^*=(-\nad{
s_2},\nad{s_1})$ for each $s=(s_1,s_2)\in S.$
\end{cor}

\begin{proof}
Let $S=S^{\pm}_4.$ Obviously, we have that
$$e_3\pod\pa P=e_{31}\frac{\pa P}{\pa x_1}+e_{32}\frac{\pa P}{\pa x_2}=\pm 2(\frac{\pa P_2}{\pa \nad z},-\frac{\pa P_1}{\pa z}).$$
Putting $P^k_j=(\frac{z^j \nad z^{k-j}}{j!(k-j)!}, 0)$ and $Q^k_j=(0, \frac{ z^j\nad z^{k-j}}{j!(k-j)!}),$
we get thus that
$$(e_3\pod\pa)^{2s}P^k_j=(-1)^s2^{2s}P^{k-2s}_{j-s},\ \ \ (e_3\pod\pa)^{2s}Q^k_j=(-1)^s2^{2s}Q^{k-2s}_{j-s},$$
$$(e_3\pod\pa)^{2s+1}P^k_j=\mp(-1)^s2^{2s+1}Q^{k-(2s+1)}_{j-s-1},\ \ \ (e_3\pod\pa)^{2s+1}Q^k_j=\pm(-1)^s2^{2s+1}P^{k-(2s+1)}_{j-s}.$$
Using these relations it is easy to obtain the explicit formulae for
$f^{k,\pm}_j.$ Obviously, the statements ($b$) and ($c$) can be verified directly using these explicit formul\ae. On the other hand, the property ($b$) follows also from the following formula
$$\frac{\pa\ }{\pa x_3}(e^{x_3e_3\pod\pa}P)=e^{x_3e_3\pod\pa}(e_3\pod\pa P)$$
and from the fact that the derivatives $\pa/\pa z$ and $\pa/\pa\nad z$ both commute with the CK extension operator $e^{x_3e_3\pod\pa}.$
\end{proof}

\begin{rem}
It is easy to express the elements $f^{k,\pm}_j$ of the GT bases from Corollary \ref{corGT3} in terms of hypergeometric series $\leftsub{2}{F_1}$ or Jacobi polynomials, see \cite[pp.\ 64 and 99]{AAR}. Indeed, we have that
\begin{eqnarray*}
p^k_j &=& \leftsub{2}{F_1}(-j,-k+j,\frac 12; -\frac{x_3^2}{|z|^2})\; \frac{z^j \nad z^{k-j}}{j!(k-j)!},\\
q^k_j &=& \leftsub{2}{F_1}(-j+1,-k+j,\frac 32; -\frac{x_3^2}{|z|^2})\; \frac{2x_3\;z^{j-1} \nad z^{k-j}}{(j-1)!(k-j)!}.
\end{eqnarray*}
Here $|z|^2=z\nad z$ and the hypergeometric series $\leftsub{2}{F_1}(a,b,c; y)$ is given by
$$\leftsub{2}{F_1}(a,b,c; y)=\sum_{s=0}^{\infty}\frac{(a)_s(b)_s}{(c)_s s!}y^s.$$
\end{rem}

In Figure \ref{fig2}, structural properties of the GT basis in this case are shown.
In the $k$-th column of Figure \ref{fig2}, the decomposition of
the $\so(3)$-module $$\cM_k=\cM_k(\bR^3,S)$$ into irreducible
$\so(2)$-submodules can be found. Moreover, all irreducible
$\so(2)$-submodules with the same highest weight are contained in
the row labeled by this highest weight. By Theorem \ref{GT3}, it is
easy to see that Figure \ref{fig2} is, in an obvious sense, composed of the
diagrams for $S^+$ and $S^-$-valued polynomials in $\bR^2$ (see
Figure \ref{fig1}). By Corollary \ref{corGT3}, we know that the application of the derivative $\pa/\pa
x_3$ to the elements of the GT basis causes the shift in the given
row to the left, the derivative $\pa/\pa\nad z$ moves them diagonally downward and $\pa/\pa z$ diagonally
upward.
In other words, the GT bases in this case possess an Appell property not only w.r.t.\ the last real variable $x_3$ but also w.r.t.\ the complex variables $z$ and $\nad z.$
Moreover, the upper triangle in Figure \ref{fig2} is mapped onto the lower one
by the transformation $(\cdot)^*.$

\begin{figure}[htb]
\centerline{ \xymatrix{
         &\cM_0 & \cM_1 &\cM_2 & \cdots\\
\frac 52&          && \lz f^2_0 \pz \ar[dl]_{\frac{\pa\ }{\pa \nad z}} \\
\frac 32&     & \lz f^1_0 \pz\ar[dl] & \lz f^2_1 \ar[l]\ar[dl]\pz & \cdots\\
\frac 12 & \lz f^0_0 \pz & \lz f^1_1 \ar[l]_{\ \frac{\pa\ }{\pa x_3}}\ar[dl] \pz & \lz f^2_2 \ar[l]\ar[ul]\ar[dl]\pz \\
-\frac 12& \lz f^0_1 \pz & \lz f^1_2 \ar[l]^{\ \frac{\pa\ }{\pa x_3}}\ar[ul]\pz & \lz f^2_3 \ar[l]\ar[dl]\ar[ul] \pz & \cdots\\
-\frac 32&    & \lz f^1_3 \pz\ar[ul] & \lz f^2_4 \ar[l]\ar[ul]\pz \\
-\frac 52&         && \lz f^2_5 \pz\ar[ul]^{\frac{\pa\ }{\pa z}} &
\cdots }}
\caption{The decomposition of the modules $\cM_k=\cM_k(\bR^3,S)$.}
\label{fig2}
\end{figure}

\begin{rem}\label{RemarkOperators}
 Let $S=S^{\pm}_4.$ It is not difficult to find non-zero constants $d^{k,\pm}_j$ such that
the polynomials $\hat f^{k}_j=d^{k,\pm}_jf^{k,\pm}_j$ satisfy the following properties
\begin{equation}\label{appell3}
\hat f^{k}_0=\nad z^k v^+,\ \hat f^{k}_{2k+1}= z^k v^-\text{\ and\ }
\frac{\pa \hat f^{k}_j}{\pa x_3}
=\left\{
\begin{array}{ll}
k\; \hat f^{k-1}_{j-1},&\ \ \ j=1,\ldots,2k;\medskip\\{}
0,&\ \ \ j=0,2k+1.
\end{array}
\right.
\end{equation}
Indeed, it is sufficient and necessary to put, for each $j=0,\ldots,k,$
$$d^{k,\pm}_j=(\mp 1)^j(-1)^{(j+1)j/2}\;2^{-j}\;k!\mand d^{k,\pm}_{2k+1-j}=(-1)^j d^k_j.$$
Moreover, we have obviously that
\begin{equation}
\hat f^{k}_{2k+1-j}=(\hat f^{k}_j)^*,\ \ \
\frac{\pa\hat f^k_j}{\pa z}=a^{k,\pm}_j \hat f^{k-1}_{j-2}
\text{\ \ \ and\ \ \ }
\frac{\pa\hat f^{k}_j}{\pa \nad z}=b^{k,\pm}_j\hat f^{k-1}_{j}
\end{equation}
where the constants $a^{k,\pm}_j$ and $b^{k,\pm}_j$ are given by
$$
a^{k,\pm}_j
= \left\{
\begin{array}{ll}
\hfill 0,& j=0,1;\medskip\\{}
-\frac 14 k,& 2\leq j\leq k;\medskip\\{}
\mp\frac 12 k,& j=k+1;\medskip\\{}
\hfill k,& k+2\leq j\leq 2k+1;
\end{array}
\right.
b^{k,\pm}_j
= \left\{
\begin{array}{ll}
\hfill k,& 0\leq j\leq k-1;\medskip\\{}
\pm\frac 12 k,& j=k;\medskip\\{}
-\frac 14 k,& k+1\leq j\leq 2k-1;\medskip\\{}
\hfill 0,& j=2k,2k+1
\end{array}
\right.
$$
Furthermore, by the definition of GT bases and their structural properties shown in Figure \ref{fig2}, it is clear that, for $k\in\bN_0,$ the sets
$$\{\hat f^{k}_j|\ j=0,\ldots,2k+1\}$$ are the GT bases of the modules
$\cM_k(\bR^3,S),$ uniquely determined by the property (\ref{appell3}) and the condition that,
for $j=0,\ldots,2k+1,$
$$H\hat f_j^k=(k+\frac 12-j)\hat f_j^k\text{\ \ \ with\ \ \ }H=-iL_{12}.$$
 \end{rem}

\section{Quaternion valued polynomials in $\bR^3$}

In this section, we reformulate the GT bases obtained in the previous section for quaternion valued spherical monogenics.

\paragraph{Quaternionic formulation}

In what follows, $\bH$ stands for the skew field of real quaternions
$q$ with the imaginary units $i_1,$ $i_2$ and $i_3,$ that is,
$$i_1^2=i_2^2=i_3^2=i_1i_2i_3=-1\text{\ \ and\ \ }q=q_0+q_1i_1+q_2i_2+q_3i_3, (q_0,q_1,q_2,q_3)\in\bR^4.$$
For a~quaternion $q,$ put $\nad q=q_0-q_1i_1-q_2i_2-q_3i_3.$
We realize $\bH$ as the subalgebra of complex
$2\times 2$ matrices of the form
\begin{equation}\label{Hmatrix}
q=\begin{pmatrix}
q_0+iq_3 & -q_2+iq_1\\
q_2+iq_1 & q_0-iq_3
\end{pmatrix}.
\end{equation}
In particular, we have that
$$i_1=\begin{pmatrix}
0 & i\\
i & 0
\end{pmatrix},\ \ \
i_2=\begin{pmatrix}
0 & -1\\
1 & 0
\end{pmatrix}\text{\ \ \ and\ \ \ }
i_3=\begin{pmatrix}
i & 0\\
0 & -i
\end{pmatrix}.$$
If $s=(q_0+iq_3, q_2+iq_1)\in\bC^2,$ then we write $q(s)$ for the quaternion $q$ as in
(\ref{Hmatrix}). For $s=(s_1,s_2)\in\bC^2,$ $q(s)$ is thus the $2\times 2$ matrix which has
$s$ as the first column and $s^*=(-\nad{
s_2},\nad{s_1})$ as the second one.
It is easy to see that $q(s)\;i_2=q(s^*)$ and that $$q(s)=\Re s_1+i_1\Im s_2+i_2\Re s_2+i_3\Im s_1$$ where,
for a~complex number $z,$ we write $\Re z$ for its real part and $\Im z$ for its imaginary part.

Furthermore, we identify $\so(3)$ with $\lz i_1,i_2,i_3\pz$ as follows: $e_{12}\simeq i_3,$ $e_{23}\simeq i_1$ and $e_{31}\simeq i_2.$
Then we realize the basic spinor representation $S$ as the space $\bC^2$ of
column vectors
$$s=\begin{pmatrix}
q_0+iq_3 \\
q_2+iq_1
\end{pmatrix}
.$$
Here the action of $\so(3)$ on $S$ is given by the matrix multiplication from
the left.

Now we are interested in quaternion valued polynomials $Q=Q(y)$ in the
variable $y=(y_0,y_1,y_2)$ of $\bR^3.$
Let us denote by $\cM_k(\bR^3,\bH)$ the space of $\bH$-valued $k$-homogeneous
polynomials $Q$ satisfying the Cauchy-Riemann
equation $DQ=0$ with
$$D=\frac{\pa\ }{\pa y_0}+i_1\frac{\pa\ }{\pa y_1}+i_2\frac{\pa\ }{\pa y_2}.$$
We can consider naturally $\cM_k(\bR^3,\bH)$ as a~right $\bH$-linear
Hilbert space with the $\bH$-valued inner product
$$(Q,R)_{\bH}=\int_{S^2}\nad QR\;d\sigma.$$
Moreover, we can identify $\cM_k(\bR^3,\bH)$ with the $\so(3)$-module $\cM_k(\bR^3,S)$
we have studied in the previous paragraph as follows.
Let $P=P(x)$ be an $S$-valued polynomial in the
variable $x=(x_1,x_2,x_3)$ of $\bR^3.$
We define a~corresponding $\bH$-valued polynomial $Q(P)$ in $\bR^3$ by
\begin{equation}\label{ident}
Q(P)(y_0,y_1,y_2)=q(P)(-y_2, y_1, y_0).
\end{equation}
Then it is easy to see that $Q(P)\in\cM_k(\bR^3,\bH)$ if and only if
$$i_1\frac{\pa P }{\pa x_1}+i_2\frac{\pa P }{\pa x_2}+i_3\frac{\pa P }{\pa x_3}=0,$$
that is, $P\in\cM_k(\bR^3,S).$
In addition, for each $P,R\in\cM_k(\bR^3,S),$ we have that
\begin{equation}\label{Hproduct}
(Q(P),Q(R))_{\bH}=q((P,R)_1,(P^*,R)_1)
\end{equation}
where $(\cdot,\cdot)_1$ is the complex valued inner product defined
as in (\ref{L2product}).
Using the identification (\ref{ident}) and Theorem \ref{GT3}, we
obtain easily orthogonal bases of quaternion valued spherical
monogenics.

\begin{thm}\label{appell}
For each $k\in\bN_0,$ there exists an orthogonal basis
\begin{equation}\label{Hbases}
\{g^k_j|\ j=0,\ldots,k\}
\end{equation}
of the right $\bH$-linear
Hilbert space $\cM_k(\bR^3,\bH)$ such that:

\medskip\noindent
(i) For $j=0,\ldots,k,$ let $h^k_j$ and $h^k_{2k+1-j}$ be the first and the second column of the (matrix valued) polynomial $g^k_j,$
respectively. Then, for each $j=0,\ldots2k+1,$ we have that
\begin{equation}\label{eigen}
Hh_j^k=(k+\frac 12-j)h^k_j\text{\ \ \ with\ \ \ }H=-i(\frac{i_3}{2}+y_2\frac{\pa\ }{\pa y_1}-y_1\frac{\pa\ }{\pa y_2}).
\end{equation}

\noindent
(ii) We have that $$\frac{\pa g^k_j}{\pa y_0} =\left\{
\begin{array}{ll}
k g^{k-1}_{j-1},&\ \ \ j=1,\ldots,k;\medskip\\{} 0,&\ \ \
j=0.
\end{array}
\right. $$

\noindent
(iii) For each $k\in\bN_0,$ we have that $g^k_0=(y_1-i_3y_2)^k.$

\medskip\noindent
Moreover, the polynomials $g^k_j$ are determined uniquely by the conditions (i), (ii) and (iii).

\medskip\noindent
In addition, for each $k\in\bN_0,$ the polynomials $$h^k_0,\ h^k_1,\ldots, h^k_{2k+1}$$ form a~GT basis of the $\so(3)$-module $\tilde\cM_k(\bR^3,S)$ of $S$-valued $k$-homogeneous polynomials $h$ in $\bR^3$ satisfying the Cauchy-Riemann equation $Dh=0.$
Moreover, the polynomials $h^k_j$ are determined uniquely by the condition \eqref{eigen}, by the Appell property
\begin{equation}\label{APh}
\frac{\pa h^k_j}{\pa y_0} =\left\{
\begin{array}{ll}
k h^{k-1}_{j-1},&\ \ \ j=1,\ldots,2k;\medskip\\{}
0,&\ \ \ j=0, 2k+1;
\end{array}
\right.
\end{equation}
and by the condition that $h^k_0=(\nad u^k,0)$ and $h^k_{2k+1}=(0, u^k)$
with $u=y_1+iy_2$ and $\nad u=y_1-iy_2.$
\end{thm}

\begin{proof}
(a) We first construct GT bases of $S$-valued monogenic polynomials in $\bR^3$
by applying Theorem \ref{GT3}.
Indeed, for $P\in\cM_k(\bR^3,S),$ we have that
$$i_2\frac{\pa P }{\pa x_1}-i_1\frac{\pa P}{\pa
x_2}=2(-\frac{\pa P_2 }{\pa \nad z}, \frac{\pa P_1}{\pa z}).$$
As in the proof of Corollary \ref{corGT3}, we get easily that the
set
$$\{f^{k,-}_0, \ldots, f^{k,-}_{2k+1}\}$$ is a~GT basis of
$\cM_k(\bR^3,S).$

\medskip\noindent
(b) For each $k\in\bN_0$ and $j=0,\ldots,2k+1,$ put
$$\hat h^k_j(y_0,y_1,y_2)=(f^{k,-}_j)(-y_2, y_1, y_0).$$
Obviously, the set $$\{\hat h^k_j|\ j=0,\ldots,2k+1\}$$ is a~GT basis of the module
$\tilde\cM_k(\bR^3,S).$ It is easy to see that
$$\hat h^k_{2j}=(-1)^{k-j}i^k(p^k_j,-iq^k_j) \mand \hat h^k_{2j+1}=(-1)^{k-j}i^k(-iq^k_{j+1},p^k_j)$$
where $p^k_j=p^k_j(u,\nad u, y_0)$ and $q^k_j=q^k_j(u,\nad u, y_0)$ are defined as in Corollary \ref{corGT3}.

\medskip\noindent
(c) We can find non-zero complex numbers
$c^k_j\in\bC$ such that the polynomials $h^k_j=c^k_j\hat h^k_j$ satisfy, in addition, the condition \eqref{APh},
$h^k_0=(\nad u^k,0),$ $h^k_{2k+1}=(0, u^k)$ and
$h^k_{2k+1-j}=(h^k_j)^*.$
Indeed, for each $k\in\bN_0,$ put $c^k_0=i^kk!.$
Moreover, it is easy to see that $$\frac{\pa\hat h^k_j}{\pa y_0}=(-1)^j\;2\hat h^{k-1}_{j-1}.$$
This implies that we need to have $c^k_j=(-1)^j 2^{-1}kc^{k-1}_{j-1}.$ Hence
we are forced to put, for each $j=0,\ldots,k,$
$$c^k_j=(-1)^{(j+1)j/2}\;2^{-j}\;k!\;i^{k-j}\mand c^k_{2k+1-j}=(-1)^j\nad c^k_j.$$

\medskip\noindent
(d) Finally, for each $k\in\bN_0$ and $j=0,\ldots,k,$ define an $\bH$-valued polynomial $g^k_j$ corresponding to the $S$-valued polynomial $h^k_j$ by
$$g^k_j=q(h^k_j).$$
By (c) and (\ref{Hproduct}), we have that the set
$$\{g^k_j|\ j=0,\ldots,k\}$$
is orthogonal with respect to the $\bH$-valued inner product $(\cdot,\cdot)_{\bH}.$
Actually, this set is, in fact, a~basis of the right $\bH$-linear
Hilbert space $\cM_k(\bR^3,\bH)$
because $$g^k_j\;i_2=q(h^k_j)\;i_2=q((h^k_j)^*)=q(h^k_{2k+1-j}).$$
Obviously, the conditions (i), (ii) and (iii) are satisfied.

\medskip\noindent
(e) Since weight vectors of the operator $H$ are determined uniquely up to non-zero multiples the construction gives also the uniqueness of the bases satisfying the conditions (i), (ii) and (iii).
\end{proof}

From the proof of Theorem \ref{appell} we get easily the next result.

\begin{cor}\label{corappell}
Let the set $\{g^k_j|\ j=0,\ldots,k\}$
be the orthogonal basis of the right $\bH$-linear
Hilbert space $\cM_k(\bR^3,\bH)$ as in Theorem \ref{appell}. Then, for each $j=0,\ldots,k,$ we have that
$$
g^k_{j}=\left\{
\begin{array}{ll}
(-1)^l\;k!\;2^{-j}(\Re p^k_l-i_1\Re q^k_l+i_2\Im q^k_l+i_3\Im p^k_l),&\ \ j=2l;\medskip\\{}
(-1)^l\;k!\;2^{-j}(\Re q^k_{l+1}+i_1\Re p^k_l-i_2\Im p^k_l+i_3\Im q^k_{l+1}),&\ \ j=2l+1.
\end{array}
\right.$$
Here $u=y_1+iy_2,$ $\nad u=y_1-iy_2$ and $p^k_j=p^k_j(u,\nad u, y_0),$ $q^k_j=q^k_j(u,\nad u, y_0)$ are complex polynomials defined as in Corollary \ref{corGT3}.
\end{cor}

\begin{rem}
In \cite{lavSL2}, the GT bases for this case are obtained in quite a~different way. In particular,
the elements $g^k_j$ of these bases are expressed in terms of the Legendre polynomials as follows.
Using spherical co-ordinates
\begin{equation*}
y_0=r\cos\theta,\ \
y_1=r\sin\theta\cos\fai,\ \
y_2=r\sin\theta\sin\fai
\end{equation*}
with $0\leq r,$ $-\pi\leq\fai\leq \pi$ and $0\leq\theta\leq\pi,$ we have namely that
$$
g^k_{j}(r,\theta,\fai)=(k!/j!)(-2)^{k-j}r^k\;(g^k_{j,0}+g^k_{j,1}\;i_1+g^k_{j,2}\;i_2+g^k_{j,3}\;i_3)\text{\ \ where}
$$
$$
\begin{array}{ll}
g^k_{j,0}=P^{j-k}_k(\cos \theta)\cos (j-k)\fai,&
g^k_{j,1}=-j P^{j-k-1}_k(\cos \theta) \cos (j-k-1)\fai,\medskip\\{}
g^k_{j,2}=j P^{j-k-1}_k(\cos \theta)\sin (j-k-1)\fai,&
g^k_{j,3}=P^{j-k}_k(\cos \theta)\sin (j-k)\fai.
\end{array}
$$
Here $P^0_k$ is the $k$-th Legendre polynomial and $P^l_k$ are its associated Legendre functions.
\end{rem}

In the last paragraph, we show that the GT bases obtained for quaternion valued spherical monogenics
coincide with those constructed by the first and the second author in \cite{BG}.

\paragraph{Identification of the bases}

The condition (ii) of Theorem \ref{appell} tells us just that the monogenic polynomials $g^k_j$
form an Appell system. In \cite{Bock2009} and \cite[Theorem 7.2]{BG}, an orthogonal Appell system
of quaternion valued spherical monogenics has been recently constructed quite explicitly from an
orthogonal system of real valued spherical harmonics. Further, in \cite{Bock2009} and
\cite{Bock2010a}, very compact recursion formulae have been obtained for the elements of the Appell
basis. From these recursion formulae it becomes also already visible that the wanted Appell system can be constructed without starting with spherical harmonics.
These results are resumed in the following theorem:

\begin{thm}[\cite{Bock2009,BG,Bock2010a}]\label{BG}
The system of inner solid spherical monogenics $\bigl\{
A_{n}^{l}\,:\,l=0,\ldots,n\bigr\}_{n\in\INo}$, where, for each $n\in\IN$ and $l=0,\ldots,n$, the
elements are given by the two-step recurrence formula
\begin{equation}\label{Equation::RecurrenceAppell_Formel_III}
A_{n+1}^{l} = \frac{n+1}{2(n-l+1)(n+l+2)}\left[ \Bigl((2n+3)y + (2n+1)\bar{y}\Bigr)A_{n}^{l} -
2n\,y\overline{y}\,A_{n-1}^{l} \right]
\end{equation}
with
\begin{equation*}
A_{l+1}^{l} \;=\; \frac{1}{4}\bigl[ (2l+3)y + (2l+1)\bar{y}\bigr]\,A_{l}^{l}\quad\text{and}\quad
A_{l}^{l} \;=\; (y_{1} - i_{3}y_{2})^{l}\,,
\end{equation*}
is an orthogonal Appell basis in $L^{2}(\bB_3,\IH) \cap \ker D$ such that for each $n\in\IN$
\begin{equation*}
\overline{D}_{0} A_{n}^{l} =\left\{
\begin{array}{ccl}
n\,A_{n-1}^{l} & : & l=0,\ldots,n-1\\
0 & : & l=n
\end{array} \right.
\end{equation*}
and
\begin{equation*}
D_{\sIC} A_{n}^{n}  = n\,A_{n-1}^{n-1}
\end{equation*}
hold. Here, $y := y_{0} + i_{1}y_{1} + i_{2}y_{2}$ denotes the reduced quaternion. The used
Cauchy-Riemann operators are defined by
$\overline{D}_{0}:=\frac{1}{2}\left(\frac{\partial}{\partial y_{0}} - i_{1}\frac{\partial}{\partial
y_{1}} - i_{2}\frac{\partial}{\partial y_{2}} \right)$ and $D_{\s\IC} :=
\frac{1}{2}\left(\frac{\partial}{\partial y_{1}} + i_{3}\frac{\partial}{\partial y_{2}}\right)$.
\end{thm}
At this point, let us remark on some structural properties of the Appell system
(\ref{Equation::RecurrenceAppell_Formel_III}) coming from a very analytical point of view. Firstly,
the two-step recurrence formulae relate Appell polynomials of different degree $n$ however the
index $l$ is fixed. Referring to Figure \ref{Figure::AppellSet}, this structurally means that the
elements of the $(l+1)$-th column are recursively generated by the initial elements $A_{l}^{l}$
which are in fact belonging to the subset of the so-called hyperholomorphic constants. Such
generalized constants are characterized in a quite natural way: A function $f$ is called
hyperholomorphic constant if $f$ belongs to the considered function space $f\in\ker D$ (the space
of monogenic solutions to the Moisil-Teodorescu system) and vanishes after (hypercomplex)
derivation. In this context, we refer again to \cite{Malonek1987} and \cite{Gurlebeck1999}, wherein
the authors have proved that the operator $\overline{D}_{0} = \frac{1}{2} \overline{D}$ corresponds
to the concept of the hypercomplex derivative. Thus a hyperholomorphic constant is analogously
characterized as in the complex one-dimensional case by $f\in \ker \overline{D}_{0} \cap \ker D$.
Secondly, Figure \ref{Figure::AppellSet} further illustrates the action of the differential
operators on the Appell basis $(\ref{Equation::RecurrenceAppell_Formel_III})$.
\begin{figure}[htb]
\begin{center}
\includegraphics[scale=1.6,angle=0]{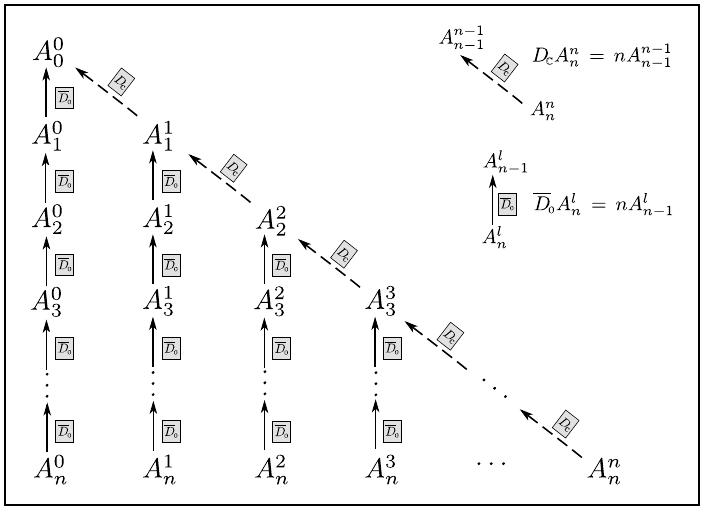}
\caption{ Structural properties of the orthogonal Appell basis $A_{n}^{l}$. }
\label{Figure::AppellSet}
\end{center}
\end{figure}
Precisely, the application of the hypercomplex derivative $\overline{D}_{0}$ to an arbitrary Appell
polynomial $A_{n}^{l}$ causes a shifting of the degree in a fixed column $l$ whereas the
application of the lower dimensional (complex) derivative $D_{\sIC}$ causes a shifting of the
degree as well as a shifting of the column. Here, it should be emphasized that the action of the
differential operator $D_{\sIC}$ is restricted to the set of hyperholomorphic constants and thus,
referring to Figure \ref{Figure::AppellSet}, is mapping along the upper diagonal. As a consequence
of the afore said, one can conclude that for an arbitrary Appell polynomial $A_{n}^{l}$,
$l=0,\ldots,n$, $n\in\INo$ of the system (\ref{Equation::RecurrenceAppell_Formel_III}) first the
$(n-l)$-fold application of $\overline{D}_{0}$ and afterwards the $l$-fold application of
$D_{\sIC}$ yields
\begin{equation*}
D_{\sIC}^{\,l}\,\overline{D}_{0}^{n-l}\,A_{n}^{l} = n!.
\end{equation*}
This property essentially enables the definition of a new Taylor series expansion (see section
\ref{applications}) in terms of the Appell set (\ref{Equation::RecurrenceAppell_Formel_III}) at
first introduced in \cite{Bock2009,BG}. Finally, it is easy to see that the system from Theorem
\ref{BG} satisfies the conditions (i), (ii) and (iii) of Theorem \ref{appell}. Hence using the GT
approach and Theorem \ref{appell} based on it, it is possible to show that $g^k_j=A^{k-j}_k$ for
all $k,j$.

\section{Orthogonal power series expansions}\label{applications}

In view of some practical application of the basis, in \cite{Bock2009,BG}, the latter basis was
particularly used to define a new Taylor series expansion 
which is a~direct consequence of the Appell property of the basis:

\begin{Def}[Taylor series in $L_{2}(\bB_3,\IH) \cap \ker D$]\label{T_series_quat}
Let $f\in L_{2}(\bB_3,\IH) \cap \ker D$. The series representation
\begin{equation}\label{Equation::TaylorSeries}
f := \sum_{n=0}^{\infty}\sum_{l=0}^{n} A_{n}^{l} \mathbf{t}_{n,l},\quad\text{with}\quad
\mathbf{t}_{n,l} = \frac{1}{n!}\,D_{\sIC}^{\,l}\,\overline{D}_{0}^{n-l}\,f(y)\,\Bigl|_{y=0}
\end{equation}
is called generalized Taylor series in $L_{2}(\bB_3,\IH) \cap \ker D$. The notations
$\overline{D}_{0}^{k}$ and $D_{\sIC}^{\,k}$ indicate the $k$-fold application of the corresponding
differential operators ($k\in\IN$) and the corresponding identity operator ($k=0$), respectively.
\end{Def}
We observe that the Taylor coefficients are given by successive application of the hypercomplex
derivative $\overline{D}_{0}$ to the principal part of the monogenic function and the "complex"
derivative $D_{\sIC}$ to the "constant" part (the subset of hyperholomorphic constants) of the
monogenic function. This Taylor series expansion meets exactly the concept of hypercomplex
derivability and improves Fueter's approach which is based on partial derivatives with respect to
the real variables $x_1$ and $x_2$.

Similarly, in case of spinor valued functions, using again the Appell property of the corresponding GT basis (see Remark \ref{RemarkOperators} at the end of Section 4) we can define the following Taylor series expansion:

\begin{Def}[Taylor series in $L_{2}(\bB_3,S) \cap \ker \pa$]\label{T_series_spin}
Let $f\in L_{2}(\bB_3,S) \cap \ker \pa$. The series representation
\begin{equation}\label{eTaylor_spinor}
f = \sum_{k=0}^{\infty}\sum_{j=0}^{2k+1} \mathbf{t}^{k}_j \;\hat f^{k}_j
\end{equation}
with the complex coefficients $\mathbf{t}^{k}_j$ such that
\begin{eqnarray*}
\mathbf{t}^{k}_j\;v^+ &=& \frac{1}{k!}\,\frac{\pa^k f(x)}{\pa x_3^j\;\pa \nad
z^{k-j}}\,\Bigl|_{x=0}\text{\ \ for\ \ }j=0,\ldots,k;\medskip\\{} \mathbf{t}^{k}_j\; v^- &=&
\frac{1}{k!}\,\frac{\pa^k f(x)}{\pa x_3^{2k+1-j}\;\pa z^{j-k-1}}\,\Bigl|_{x=0}\text{\ \ for\ \
}j=k+1,\ldots,2k+1.
\end{eqnarray*}
is called generalized Taylor series in $L_{2}(\bB_3,S) \cap \ker \pa$.
\end{Def}
Let us note that the partial derivatives $\pa/\pa x_3,$ $\pa/\pa z$ and $\pa/\pa \nad z$ commute
with each other.


It is interesting to compare both Taylor series from Definitions \ref{T_series_quat} and
\ref{T_series_spin}. In both cases, the basis is orthogonal and the corresponding coefficients can
be expressed using (linear combinations of) partial derivatives of the corresponding function. The
derivatives used in both cases look different  but there are trivially related (at least for
monogenic functions) to each other. In the formulation using spinor valued functions, the Appell
property is true even w.r.t.\ all three variables. Hence in this case application of any of three
basic derivatives map any basis element to a multiple of another basis element. For quaternion
valued functions, it is not the case.

Applying a simple normalization (see, i.e., \cite{Bock2009,BG}) to each element
(\ref{Equation::RecurrenceAppell_Formel_III}) of the Appell basis, explicitly given by the relation
\begin{equation}\label{Equation::Transformation_A_Phi}
\varphi_{n,\sIH}^{l} \;=\;
\frac{1}{2^{l+1}\,n!}\,\sqrt{\frac{(2n+3)\,(n-l)!\,(n+l+1)!}{\pi}}\;A_{n}^{l},\;\;
l=0,\ldots,n,\;n\in\INo,
\end{equation}
yields directly:
\begin{cor}[\cite{Bock2009,BG}]
The system of inner solid spherical monogenics
\begin{equation}\label{Equation::CONS_IH}
\bigl\{ \varphi_{n,\sIH}^{l}\,:\,l=0,\ldots,n\bigr\}_{n\in\INo}
\end{equation} is an orthonormal basis in
$L^{2}(\bB_3,\IH) \cap \ker D$.
\end{cor}
Due to the orthogonality and the completeness of the orthonormal system (\ref{Equation::CONS_IH})
we state the Fourier series expansion in $L_{2}(\bB_3,\IH) \cap \ker D$.
\begin{cor}[Fourier series in $L_{2}(\bB_3,\IH) \cap \ker D$]
Let $f \in L_{2}(\bB_3,\IH) \cap \ker D$. Then $f$ can be uniquely represented in terms of the
orthonormal system (\ref{Equation::CONS_IH}), that is:
\begin{equation}\label{Equation::FourierSeries}
f\,:=\, \sum_{n=0}^{\infty}\sum_{l=0}^{n}\,\varphi_{n,\sIH}^{l}\,
\boldsymbol{\alpha}_{n,l},\quad\text{with}\quad\boldsymbol{\alpha}_{n,l}\,=\,\int_{\mathbb{B}_{3}}
\overline{\varphi_{n,\sIH}^{l}}\,f\,d\la^3.
\end{equation}
\end{cor}
Here, it should be emphasized that in contrast to the complex case the order of
$\varphi_{n,\sIH}^{l}$ and $f$ in the inner products has to be respected. As a direct consequence
of relation (\ref{Equation::Transformation_A_Phi}) and the orthogonality of both series expansions,
each Fourier coefficient (\ref{Equation::FourierSeries}) of a function $f\in
L_{2}(\bB_3,\IH)\cap\ker D$ can be explicitly expressed in terms of the corresponding Taylor
coefficient (\ref{Equation::TaylorSeries}) and vice versa by
\begin{equation*}
\boldsymbol{\alpha}_{n,l} \;=\; 2^{l+1}\,\sqrt{\frac{\pi}{(2n+3)\,(n-l)!\,(n+l+1)!}}\;
D_{\sIC}^{\,l}\,\overline{D}_{0}^{n-l}\,f(\mathbf{x})\,\Bigl|_{\mathbf{x}=\mathbf{0}},
\end{equation*}
where $l=0,\ldots,n$ and $n\in\INo$. This important analytic property of the series expansions
analogously corresponds to the complex one-dimensional case.

\section*{Acknowledgments}

R. L\'avi\v cka and V. Sou\v cek acknowledge the financial support from the grant GA 201/08/0397.
This work is also a part of the research plan MSM 0021620839, which is financed by the Ministry of Education of the Czech Republic.












\bigskip\bigskip\bigskip

\noindent
Sebastian Bock and Klaus G\"urlebeck,\\ Institute of Mathematics and Physics, Bauhaus University,
Weimar, Germany\\
email: \texttt{sebastian.bock@uni-weimar.de},\ \ \texttt{klaus.guerlebeck@uni-weimar.de}
\bigskip

\noindent
Roman L\'avi\v cka and Vladim\'ir Sou\v cek,\\ Mathematical Institute, Charles University,\\ Sokolovsk\'a 83, 186 75 Praha 8, Czech Republic\\
email: \texttt{lavicka@karlin.mff.cuni.cz},\ \ \texttt{soucek@karlin.mff.cuni.cz}

\end{document}